\documentclass{amsart}

\usepackage[margin=1in,top=1in,centering]{geometry}
\usepackage[utf8]{inputenc}
\usepackage[english]{babel}
\usepackage{mathtools}
\usepackage{graphicx}
\usepackage{amsmath,amssymb,amsthm}
\usepackage{comment}
\usepackage{enumitem}
\usepackage[normalem]{ulem}
\usepackage{leftidx}
\usepackage{tikz}
\usetikzlibrary{cd}

\usepackage[
colorlinks=true,
urlcolor=purple,
linkcolor=purple!87!black,
pdfborder={0 0 0}
]{hyperref}

\renewcommand{\phi}{\varphi}

\newcommand{\findim}{\mathrm{findim}\,}
\newcommand{\Findim}{\mathrm{Findim}\,}
\newcommand{\dell}{\mathrm{dell}\,}
\newcommand{\phidim}{\varphi\dim}
\renewcommand{\dell}{\mathrm{dell}\,}
\newcommand{\rad}{\mathrm{rad}\,}
\newcommand{\add}{\mathrm{add}\,}
\newcommand{\pd}{\mathrm{pd}\,}
\newcommand{\gldim}{\mathrm{gldim}\,}
\newcommand{\id}{\mathrm{id}\,}
\newcommand{\op}{\mathrm{op}}
\renewcommand{\add}{\mathrm{add}\,}

\newcommand{\Z}{\mathbb{Z}}
\newcommand{\N}{\mathbb{N}}

\renewcommand{\mod}{\mathrm{mod}\,}
\newcommand{\Mod}{\mathrm{Mod}\,}
\newcommand{\smod}{\underline{\mathrm{mod}}\,}

\newcommand{\dirsum}{\xhookrightarrow{\!\oplus}}
\newcommand{\into}{\xhookrightarrow{}}
\newcommand{\onto}{\twoheadrightarrow}
\newcommand{\agemo}{\raisebox{\depth}{\scalebox{1}[-1]{\(\Omega\)}}}
\newcommand{\Tr}{\mathrm{Tr}\,}
\newcommand{\grade}{\mathrm{grade}\,}
\renewcommand{\top}{\mathrm{top}\,}
\newcommand{\depth}{\mathrm{depth}\,}
\newcommand{\Ext}{\mathrm{Ext}}
\newcommand{\Hom}{\mathrm{Hom}}
\newcommand{\kdell}{k\text{-}\mathrm{dell}\,}
\newcommand{\ddell}{\mathrm{ddell}\,}

\newcommand{\kddell}{k\text{-}\mathrm{ddell}\,}

\theoremstyle{plain}
\newtheorem{theorem}{Theorem}[section]
\newtheorem{definition}[theorem]{Definition}
\newtheorem{lemma}[theorem]{Lemma}
\newtheorem{proposition}[theorem]{Proposition}
\newtheorem{corollary}[theorem]{Corollary}
\newtheorem{remark}[theorem]{Remark}
\newtheorem{example}[theorem]{Example}
\newtheorem{question}[theorem]{Question}
\newtheorem{construction}[theorem]{Construction}
\newtheorem{observation}[theorem]{Observation}

\title{Symmetry of Derived Delooping Level}

\author{Ruoyu Guo$^*$}

\thanks{$^*$Department of Mathematics, Brandeis University, Waltham, MA, USA. \texttt{rguo@brandeis.edu}}



\begin{document}

\begin{abstract}
The finitistic dimension conjecture is closely connected to the symmetry of the finitistic dimension. Recent work indicates that such connection extends to one of its upper bounds, the delooping level. In this paper, we show that the same holds for the derived delooping level, which is an improvement of the delooping level. This reduces the finitistic dimension conjecture to considering algebras whose opposite algebra has (derived) delooping level zero. We thereby demonstrate ways to utilize the new concept of derived delooping level to obtain new results and present additional work involving tensor product of algebras. 
\end{abstract}

\maketitle

\section{Introduction}

The finitistic dimension conjecture, which states that the little finitistic dimension (findim) of an Artin algebra $\Lambda$ is always finite, carries significant homological implications. It is a sufficient condition for numerous other conjectures for Artin algebras, the most important of which includes the Wakamatsu tilting conjecture \cite{mantese2004wakamatsu}, the Gorenstein symmetry conjecture (a consequence of the Wakamatsu tilting conjecture), the Auslander-Reiten conjecture \cite{auslander1975generalized, happel1991homological}, and the Nakayama conjecture (a consequence of the Auslander-Reiten conjecture). Investigations of the findim conjecture go hand in hand with a better understanding of the representation theory of Artin algebras. Since the conjecture specifically asks about the projective dimension of $\Lambda$-modules, it is proved in special cases where the module category or the syzygy category is well understood.

In addition to solving the findim conjecture through a thorough understanding of the module category $\mod\Lambda$, other techniques rely on various upper and lower bounds of findim. There is too much work done on the subject to be comprehensive, so we only mention some invariants we are most interested in. One such lower bound is called the depth, defined as the supremum of $\grade S$ over all simple $\Lambda$-modules $S$, where
\[
\grade S = \inf\{n\in\N \mid \Ext_{\Lambda}^n(S, \Lambda) \neq 0\}.
\]

These definitions come from the study of stable module category \cite{auslander1969stable}, and when $\Lambda$ is a commutative Noetherian local ring, the Auslander-Buchsbaum formula \cite{auslander1957homological} implies $\depth\Lambda = \findim\Lambda$. Popular upper bounds include the $\varphi$-dimension $\phidim\Lambda$ \cite{igusa2005finitistic} and the delooping level $\dell\Lambda$ \cite{gelinas2022}. The author and Igusa \cite{guo2025derived} recently improved the delooping level to the derived delooping level $\ddell\!$. Precisely, these upper bounds satisfy
\[
\findim\Lambda \leq \phidim\Lambda,
\]
\[
\Findim\Lambda^{\op} \leq \ddell\Lambda \leq \dell\Lambda.
\]

One natural question related to these upper bounds is whether they are always finite, and a positive answer to that question would solve the findim conjecture for Artin algebras. For the rest of the paper, when we refer to the little findim conjecture, the big findim conjecture, the $\varphi$-dimension conjecture, the delooping level conjecture, and the derived delooping level conjecture, we mean their corresponding invariants $\findim\Lambda$, $\Findim\Lambda$, $\phidim\Lambda$, $\dell\Lambda$, $\ddell\Lambda$ are finite for all Artin algebras $\Lambda$, respectively. The $\phi$-dimension conjecture is false by the { counterexamples in \cite{barrios2024algebras, hanson2022counterexample}}, but the delooping level is zero in those cases. The delooping level conjecture is false by the counterexample in \cite{kershaw2023finite}. However, the derived delooping level is shown to be finite and equal to the big finitistic dimension of the opposite algebra in that case (Example 3.8 in \cite{guo2025derived}), providing evidence that this new concept deserves future attention.

Another interesting aspect of the findim conjecture involves its symmetry. Let $\Lambda$ be an Artin algebra. Cummings \cite[Theorem A]{cummings2024left} proves the equivalence between the big finitistic dimension conjecture and the statement that $\Findim\Lambda<\infty$ implies $\Findim\Lambda^{\op}<\infty$ for all $\Lambda$. She also proves the stronger result \cite[Theorem B]{cummings2024left} that if $\Findim\Lambda=\infty$, then there is a related algebra $\tilde{\Lambda}$ (using Construction \ref{con: multi-point extension} later) such that $\Findim\tilde{\Lambda}=\infty$ and $\Findim\tilde{\Lambda}^{\op}=0$. The analogous result for the delooping level is the equivalence between the delooping level conjecture and the statement that $\dell\Lambda<\infty$ implies $\dell\Lambda^{\op}<\infty$ for all $\Lambda$ and is proved in \cite{zhang2022left}. Due to the counterexample in \cite{kershaw2023finite}, we know that the delooping level conjecture does not hold. However, there is no known example where $\ddell\Lambda=\infty$. In this paper, we prove the corresponding symmetry statement for the derived delooping level, thus providing another sufficient condition for the findim conjecture.

\begin{theorem}
\label{thm:main theorem in intro}
The derived delooping level conjecture  holds if and only if $\ddell\Lambda=0$ implies $\ddell\Lambda^{\op}<\infty$ for all Artin algebras $\Lambda$.
\end{theorem}

In our formulation, algebras satisfying $\Findim\Lambda=0$ have many useful properties such as every embedding from one projective module to another splits. These properties may make the findim conjecture easier to work with, compared to the case where we consider all $\Lambda$ with $\ddell\Lambda<\infty$.

\textbf{Acknowledgements.} The author is grateful to his advisor Kiyoshi Igusa for his continuous support and helpful conversations regarding the paper. The author thanks Emre Sen for mentioning the papers \cite{cummings2024left, zhang2022left} during the 2024 Maurice Auslander Distinguished Lectures and International Conference. { The author is grateful to the anonymous referee for pointing out mistakes in Lemma \ref{lem:dell A vs dell lambda} and Proposition \ref{prop:ddell A vs ddell lambda} and for providing very helpful and constructive comments to improve the first version of the paper.}

\section{Preliminaries}

We start with the necessary definitions and notations for the rest of the paper. Let $\Lambda$ be an Artin algebra. For the sake of using quiver path algebras { and their quotients} as examples, we may also think of $\Lambda$ as a basic finite dimensional algebra. We do not lose much generality this way. Let $\mod\Lambda$ be the category of finitely generated \textbf{right} $\Lambda$-modules so that $\mod\Lambda^{\op}$ is the category of finitely generated left $\Lambda$-modules. Similarly, let $\Mod\Lambda$ and $\Mod\Lambda^{\op}$ be the category of all right and left $\Lambda$-modules, respectively. If we use the word module without specifying left or right, we always mean \textbf{right} module. We use $\prescript{}{A}{M}$ (resp. $M_A$) to mean $M$ is a left (resp. right) $A$-module. For every module $M$, we can define the \textbf{syzygy} $\Omega M$ of $M$ (resp. cosyzygy $\Sigma M$ of $M$) as the kernel of the surjection from the projective cover $P(M)\onto M$ (resp. cokernel of the injection $M\into I(M)$ into the injective envelope). The finitistic dimension conjecture states that the little finitistic dimension $\findim\Lambda$ and the big finitistic dimension $\Findim\Lambda$ are finite for all $\Lambda$, where
\[
\findim\Lambda = \sup \{\pd M\mid \pd M<\infty, M\in\mod\Lambda \},
\]
\[
\Findim\Lambda = \sup \{\pd M\mid \pd M<\infty, M\in\Mod\Lambda \},
\]
\[
\findim\Lambda^{\op} \,(\text{resp. } \Findim\Lambda^{\op}) = \sup \{\id M\mid \id M<\infty, M\in\mod\Lambda \,(\text{resp. } M\in\Mod\Lambda)\}
\]
\[
\pd M = \inf\{n\in\N\mid \Omega^n M \text{ is projective}\},
\]
\[
\id M = \inf\{n\in\N\mid \Sigma^n M \text{ is injective}\}.
\]

The two upper bounds that we investigate are the delooping level and derived delooping level. Let $M\dirsum N$ mean { $M$ is a \textbf{direct summand} of $N$. We say $M$ is a \textbf{stable retract} of $N$ if there is a split monomorphism $M\to N\oplus P$ for some projective module $P$, that is, $M\dirsum N\oplus P$.
}
Let $\agemo=\Tr\Omega\Tr$ be an endofunctor on the \textbf{stable module category} $\smod\Lambda$. It is known that $(\agemo, \Omega)$ is an adjoint pair on $\smod\Lambda$. Let $\N=\Z_{\geq 0}$ and use the convention that the infimum of the empty set is $+\infty$. We recall the following definitions.

\begin{definition}\cite{gelinas2022, guo2025derived}
\label{def: dell, kdell, ddell}
Let $M\in\smod\Lambda$ and $k$ be a positive integer.
\begin{enumerate}
{
\item The \textbf{delooping level of $M$} is
\begin{align*}
\dell M & = \inf\{n\in\N \mid \Omega^n M \text{ is a stable retract of } \Omega^{n+1}\agemo^{n+1}\Omega^n M\} \\
& = \inf\{n\in\N \mid \Omega^n M \text{ is a stable retract of } \Omega^{n+1}N \text{ for some $N$}\}.
\end{align*}
\item The \textbf{$k$-delooping level} of $M$ is
\begin{align*}
\kdell M & = \inf\{n\in\N \mid \Omega^n M \text{ is a stable retract of } \Omega^{n+k}\agemo^{n+k}\Omega^n M\} \\
& = \inf\{n\in\N \mid \Omega^n M \text{ is a stable retract of } \Omega^{n+k}N \text{ for some $N$}\}.
\end{align*}
}
In the special case that $\kdell M=0$ for all $k\in\Z_{>0}$, we say $M$ is \textbf{infinitely deloopable}.

\item The \textbf{derived delooping level} of $M$ is
\begin{align*}
\ddell M = \inf \{m\in\N \mid & \,\exists n\leq m \text{ and an exact sequence in $\mod\Lambda$ of the form} \\
& \,\, 0 \to C_n \to C_{n-1} \to \cdots \to C_1 \to C_0 \to M \to 0, \\
& \text{ where $(i+1)$-$\dell C_i\leq m-i$, } i=0,1,\dots,n  \},
\end{align*}
where we say \textbf{$\ddell M$ is equal to $m$ using $n$ and the exact sequence} $0\to C_n\to \cdots \to C_0\to M\to 0$.
\end{enumerate}
\end{definition}

\textbf{Notation.} In cases where it is helpful to point out the algebra we are working with, we write the algebra as a subscript of the invariant. For example, $\dell\!_A M$ means the delooping level of $M$ considered in $\mod A$.

Note that the equivalence between $\Omega^nM$ being a stable retract of $\Omega^{n+1}N$ for some $N$ and $\Omega^n M$ being a stable retract of $\Omega^{n+1}\agemo^{n+1}\Omega^n M$ in the definition of the delooping level is proved in \cite[Theorem 1.10]{gelinas2022}. The author and Igusa \cite{guo2025derived} extend the definition to $k$-delooping level. It is clear that $(\agemo^k, \Omega^k)$ is still an adjoint pair on $\smod\Lambda$, so the equivalence holds for general $k\in\Z_{>0}$, as seen in the second item of Definition \ref{def: dell, kdell, ddell}. The derived delooping level presented here is the special case of the more general $k$-derived delooping level $\kddell\!$ in \cite[Definition 2.22]{guo2025derived} when $k=1$. Note that the set of modules with finite derived delooping level $\ddell\Lambda=1$-$\ddell\Lambda$ is a torsion-free class in $\mod\Lambda$ \cite{guo2025derived}, a property that the set of modules with finite $\kddell\Lambda$ when $k>1$ does not necessarily have. By definition, we also have $k_1$-$\ddell\Lambda\leq k_2$-$\ddell\Lambda$ if $k_1 < k_2$, so it is often convenient to only consider the upper bound $\kddell\Lambda$ when $k=1$, as is the case here in this paper.

For the three invariants above that are defined on modules, we can define them on the algebra $\Lambda$ as the supremum of the invariant over all simple $\Lambda$-modules. For example,
\[
\ddell\Lambda = \sup\{\ddell S\mid S \text{ is simple $\Lambda$-module}\}.
\]

These invariants are all upper bounds of the big finitistic dimension of the opposite algebra.

\begin{theorem}\cite{gelinas2022, guo2025derived}
\label{thm:dell, ddell bounds Findim}
For Artin algebras $\Lambda$,
\begin{equation}
\label{eq:dell, ddell bounds Findim}
\Findim\Lambda^{\op} \leq \ddell\Lambda \leq \dell\Lambda \leq \kdell\Lambda.
\end{equation}
\end{theorem}

In many cases, the upper bounds $\dell\Lambda$ and $\ddell\Lambda$ are equal to $\Findim\Lambda^{\op}$, such as the radical square zero case \cite{gelinas2021finitistic}. While $\dell\Lambda = \Findim\Lambda^{\op}$ is not true in general (in fact, arbitrarily different as in \cite{kershaw2023finite}), { there is no conclusive example} where $\Findim\Lambda^{\op} = \ddell\Lambda$ does not hold. So, it is interesting to ask to what extent the equalities $\Findim\Lambda^{\op} = \ddell\Lambda$ and $\Findim\Lambda^{\op} = \dell\Lambda$ can hold.

In the special case where $\gldim\Lambda<\infty$, both equalities hold. On the one hand, we must have $\Findim\Lambda=\gldim\Lambda$ and they are upper bounded by $\kdell\Lambda^{\op}$ and $\ddell\Lambda^{\op}$. On the other hand, if $M$ has finite projective dimension, $\kdell M\leq \pd M$ and $\ddell M\leq \pd M$ by definition. We immediately have
\begin{equation}
\label{eq:kddell is gldim when finite}
\gldim\Lambda^{\op} \geq \kdell\Lambda^{\op} \geq \ddell\Lambda^{\op} \geq \gldim\Lambda \geq \kddell\Lambda \geq \ddell\Lambda \geq \gldim\Lambda^{\op}.
\end{equation}

Therefore, for all $k\in\Z_{>0}$, $\kdell\Lambda$ and $\ddell\Lambda$ describe $\gldim\Lambda$ and $\gldim\Lambda^{\op}$ exactly if $\gldim\Lambda$ or $\gldim\Lambda^{\op}$ is finite. This is stated as the following observation.
\begin{observation}
\label{prop:kddell is gldim when finite}
If $\gldim\Lambda<\infty$,
\[
\gldim\Lambda^{\op} = \kdell\Lambda^{\op} = \ddell\Lambda^{\op} = \gldim\Lambda = \kdell\Lambda = \ddell\Lambda
\]
for all $k\in\Z_{>0}$.
\end{observation}

In order to prove the main theorem of the paper, we recall the construction first given in \cite{cummings2024left} for finite dimensional algebras, which is then generalized to any ring in \cite{krause2022symmetry}. Cummings \cite{cummings2024left} used the construction to prove { the big finitistic dimension conjecture} is equivalent to the statement that $\Findim\Lambda<\infty$ implies $\Findim\Lambda^{\op}<\infty$.

\begin{construction}
\label{con: multi-point extension}
For any Artin algebra $A$, let $\bar{S} = \top A = A/\rad A$ and $B = T(\bar{S})$, where $T(\bar{S})$ is the trivial extension. As an abelian group, $T(\bar{S})$ is $\bar{S}\oplus \leftidx{_{\bar{S}}}{\bar{S}}{_{\bar{S}}}$, and it also follows the multiplication rule
\[
(a_1,b_1)\cdot(a_2,b_2) = (a_1a_2, a_1b_2+b_1a_2).
\]

Consider the triangular matrix algebra $\tilde{A}$ associated to an Artin algebra $A$
\begin{equation}
\label{eq:A tilde}
\tilde{A} = \begin{pmatrix}
A & 0 \\ \leftidx{_B}{\bar{S}}{_A} & B
\end{pmatrix}.
\end{equation}

Let $\mathbb{K}$ be a field. Note that if $A=\mathbb{K}Q/I$ is the quotient of a path algebra of a quiver $Q$ with relations $I$, then the quiver of $\tilde{A}$ has twice as many vertices as $Q$, and specifically attaches $\begin{tikzcd} \tilde{i} \ar[loop left, "\beta"] \ar[r,"\alpha"] & i \end{tikzcd}$ to each vertex $i$ of $Q$. It also adds the relations $\beta^2 = \beta\alpha=\alpha\cdot\rad \mathbb{K}Q = 0$.
\end{construction}

The idea to attach $\begin{tikzcd} \tilde{i} \ar[loop left, "\beta"] \ar[r,"\alpha"] & i \end{tikzcd}$ or a loop to some vertices of a quiver is not new. It showed up in \cite[Example 2.2]{jensen1982homological} as an algebra $\Lambda$ with $\findim\Lambda=1$ and $\findim\Lambda^{\op}=0$. That example is a specialization of Example \ref{ex:findim and op arbitrarily different} when $n=2$. Example \ref{ex:findim and op arbitrarily different} first appeared in \cite[Example 1.2]{green1991finitistic} as one way to construct a monomial algebra whose left and right finitistic dimensions are arbitrarily different. We present the same example below.

\begin{example}
\label{ex:findim and op arbitrarily different}
Let $Q$ be the following quiver with $n+1$ vertices.

\begin{center}
\begin{tikzcd}
1' \arrow[out=150, in=210, loop] \arrow[r] & 1 \arrow[r] & 2 \arrow[r] & \cdots \arrow[r] & n \\
\end{tikzcd}.
\end{center}

Let $\Lambda=\mathbb{K}Q/\rad\!^2\mathbb{K}Q$. It is clear that $\gldim\Lambda=\infty$, but we find that $\Findim\Lambda=\findim\Lambda=n$ is achieved by the projective resolution
\begingroup
\setlength\arraycolsep{1pt}
\[
0\to S_n \to \begin{matrix} n-1 \\ n \end{matrix} \to \cdots \to \begin{matrix} 1 \\ 2 \end{matrix} \to \begin{matrix} {} & 1' & {} \\ 1' & {} & 1 \end{matrix} \to \begin{matrix} 1' \\ 1' \end{matrix} \to 0
\]
\endgroup

On the other hand, the indecomposable injective $\Lambda$-modules are
\begingroup
\setlength\arraycolsep{1pt}
\[
\begin{matrix} 1' \\ 1' \end{matrix} \quad \begin{matrix} 1' \\ 1 \end{matrix} \quad \begin{matrix} 1 \\ 2 \end{matrix} \quad \dots \quad \begin{matrix} n-1 \\ n \end{matrix},
\]
and they have the same Lowey length 2. So there cannot be any surjection from an injective module to another that does not split. That is, $\findim\Lambda^{\op}=\Findim\Lambda^{\op}=0$.
\endgroup
\end{example}

The phenomenon in Example \ref{ex:findim and op arbitrarily different} lays the intuition for why we are able to prove this symmetry condition for the finitistic dimension, the delooping level, and the derived delooping level. 
By Construction \ref{con: multi-point extension}, to find the lower triangular matrix algebra $\widetilde{A^{\op}}$, we need to find the associated bottom left and right entries of the matrix as in \eqref{eq:A tilde}. Let $A$, $\bar{S}$, $B$ be as in Construction \ref{con: multi-point extension}. The bottom left entry $\bar{S}=\top A^{\op}$ of $\widetilde{A^{\op}}$ is thought of as a right $A^{\op}$-module, and the bottom right entry $T(\bar{S})$ as an abelian group is still $\bar{S}\oplus \leftidx{_{\bar{S}}}{\bar{S}}{_{\bar{S}}}$, but its multiplication is inherited from $A^{\op}$. Therefore, the bottom right entry is $T(\bar{S})^{\op}=B^{\op}$. This gives us
\[
\widetilde{A^{\op}} = \begin{pmatrix}
A^{\op} & 0 \\ \prescript{}{B^{\op}}{\bar{S}_{A^{\op}}} & B^{\op}
\end{pmatrix}.
\]
Defining $\Lambda^{\op} = \widetilde{A^{\op}}$, we finally get
\[
\Lambda = (\widetilde{A^{\op}})^{\op} = \begin{pmatrix}
A & \leftidx{_A}{\bar{S}}{_B} \\ 0 & B
\end{pmatrix} \cong
\begin{pmatrix}
B & 0 \\ \leftidx{_A}{\bar{S}}{_B} & A.
\end{pmatrix}
\]

We would like to know how the construction of $\Lambda$ from $A$ implies about the relationship between $\ddell A$ and $\ddell\Lambda$. First of all, it is easy to see that $\dell\Lambda^{\op}=\ddell\Lambda^{\op}=0$ since every simple module embeds in a projective module, which is also shown in \cite{cummings2024left} and \cite{krause2022symmetry}.

In \cite{Auslander_Reiten_Smalo_1995}, the authors provide the representation theory of such triangular matrix algebras. In particular, all left $\Lambda^{\op}$-modules are of the form $(\prescript{}{A^{\op}}{M}, \prescript{}{B^{\op}}{N}, f)$, where $f:\prescript{}{B^{\op}}{\bar{S}_{A^{\op}}} \otimes_{A^{\op}} M \to \prescript{}{B^{\op}}{N}$ is a morphism of left $B^{\op}$-modules. Therefore, all right $\Lambda$-modules are of the form $(M_A,N_B,f)$, where $f:M\otimes_A \bar{S} \to N$ is a morphism of right $B$-modules. This is how we will write (right) $\Lambda$-modules in the next section.

Morphisms from $(M_1,N_1,f_1)$ to $(M_2,N_2,f_2)$ in $\mod\Lambda$ are of the form $(\alpha,\beta)$, where $\alpha:M_1\to M_2$ and $\beta:N_1\to N_2$, such that the following diagram commutes:

\begin{center}
\begin{tikzcd}
M_1\otimes_A \bar{S} \arrow[r, "\alpha\otimes 1_{\bar{S}}"] \arrow[d, "f_1"] & M_2\otimes_A \bar{S} \arrow[d, "f_2"] \\
N_1 \arrow[r, "\beta"] & N_2
\end{tikzcd}.
\end{center}
{ Indecomposable projective $\Lambda$-modules are of the form $(P, P\otimes_A \bar{S}, 1_{P\otimes_A \bar{S}})$ and $(0,Q,0)$, where $P,Q$ are indecomposable projective $A$-module and $B$-module, respectively. Note that $P\otimes_A \bar{S}$ is nonzero for any projective $A$-module $P$ since $\bar{S}=\top A$. Indecomposable injective $\Lambda$-modules are of the form $(I,0,0)$ and $(\Hom_B(\bar{S}, J), J, \phi)$, where $I,J$ are injective $A$-module and $B$-module, respectively, and $\phi:\Hom_B(\bar{S}, J)\otimes_A \bar{S} \to J$ is the evaluation map $\phi(g\otimes x) = g(x)$ for $g\in \Hom_B(\bar{S}, J)$ and $x\in \bar{S}$.} Further properties of the finitistic dimensions and representation theory of triangular matrix algebras are investigated in \cite{Auslander_Reiten_Smalo_1995, fossum2006trivial}.

\section{Proof of main theorem}

Our goal is to prove Theorem \ref{thm:main theorem in intro} that investigates the symmetry of the derived delooping level. The two module categories $\mod A$ and $\mod B$ naturally embed in $\mod\Lambda$ in the following sense. Objects of the full subcategory identified with $\mod A$ (resp. $\mod B$) are of the form $(M_A,0,0)$ (resp. $(0, N_B, 0)$).

Two useful facts from \cite[Lemma 2.7]{zhang2022left} are stated in the next lemma. We rephrase the lemma slightly to accommodate for our setting.

\begin{lemma}\cite{zhang2022left}
\label{lem:SES in mod lambda and A}
Let $A$, $\bar{S}$, and $B$ be as in Construction \ref{con: multi-point extension}. Let $\Lambda=(\widetilde{A^{\op}})^{\op}$ be as above.
\begin{enumerate}
\item If 
\[
0\to (M',N',f') \to (C, D, g) \to (M, N, f) \to 0
\]
is a short exact sequence in $\mod\Lambda$, then
\[
0 \to M' \to C \to M \to 0
\]
is a short exact sequence in $\mod A$.
\item If $P$ is the projective cover of $M$ in $\mod A$, then we have an exact sequence of $\Lambda$-modules
\[
0\to (\Omega M,0,0) \oplus (0,Z,0) \to (P, P\otimes_A \bar{S}, 1) \oplus (0,Q,0) \to (M, N, f) \to 0,
\]
where $Q$ is some projective $B$-module and $Z\in\add(B\oplus \top B)$.
\end{enumerate}
\end{lemma}

The authors in \cite{zhang2022left} use the previous lemma to prove $\dell A\leq \dell\Lambda$, but in fact, { the stronger relation $\dell A\leq\dell\Lambda\leq \dell A + 1$ holds} in that case. { We show some more general statements in the following lemma.}

\begin{lemma}
\label{lem:dell A vs dell lambda}
Let $A$, $\bar{S}$, and $B$ be as in Construction \ref{con: multi-point extension}. Let $\Lambda=(\widetilde{A^{\op}})^{\op}$ be as above. Let $M\in\mod A$, $N\in\mod B$, and $f:M\otimes_A \bar{S}\to N$ be a $B$-morphism so that $(M, N, f)$ is a $\Lambda$-module. Then for all $k\in\Z_{>0}$, we have
{
\begin{enumerate}
\item $\kdell\!_A M \leq \kdell\!_{\Lambda} (M, N, f) \leq \kdell\!_A M + 1$.
\item $\kdell A \leq \kdell\Lambda \leq \kdell A + 1$.
\end{enumerate}
}
\end{lemma}

\begin{proof}
\begin{enumerate}
\item By (2) of Lemma \ref{lem:SES in mod lambda and A}, we get that $\Omega_{\Lambda} (M,N,f) = (\Omega_A M, 0, 0)\oplus (0,Z,0)$. Since $B$ is isomorphic to the direct sum of copies of $\mathbb{K}[X]/(X^2)$, every { finitely generated} $B$-module is the direct sum of copies of $\mathbb{K}$ and $\mathbb{K}[X]/(X^2)$. As $B$-modules, $\mathbb{K}$ is infinitely deloopable as its own syzygy, and $\mathbb{K}[X]/(X^2)$ is projective. Thus, every $\Lambda$-module of the form $(0,Z,0)$ is infinitely deloopable. For higher syzygies, we see
\[
\Omega_{\Lambda}^i (M,N,f) = (\Omega_A^i M,0,0) \oplus (0,Z',0),
\]
where $(0,Z',0)$ is infinitely deloopable { and therefore has $k$-delooping level 0 for all $k\in\Z_{>0}$.

We first show $\kdell\!_A M \leq \kdell\!_{\Lambda} (M, N, f)$. Suppose $\kdell\!_{\Lambda} (M,N,f)=m$ is finite. Then there exists $(M',N',f')\in\mod\Lambda$ such that
\[
\Omega_{\Lambda}^m (M,N,f) = (\Omega_A^m M, 0, 0) \oplus (0,Z,0) \dirsum \Omega_{\Lambda}^{m+k}(M',N',f') \oplus U = (\Omega_A^{m+k} M', 0, 0) \oplus (0,Z',0) \oplus U,
\]
where $U$ is some projective $\Lambda$-module and both $(0,Z,0)$ and $(0,Z',0)$ are infinitely deloopable. By the definition of morphisms in $\mod\Lambda$, there is no nonzero map from $(\Omega_A^m M, 0, 0)$ to $(0,Z',0)$. Also, the $\Lambda$-module $(\Omega_A^m M, 0, 0)$ does not have projective summands, so $(\Omega_A^m M, 0, 0)$ must be a direct summand of $(\Omega_A^{m+k} M',0,0)$. In other words, $\Omega_A^m M$ is a direct summand of $\Omega_A^{m+k} M'$ in $\mod A$, showing $\kdell\!_A M\leq m$.

Now we prove $\kdell\!_{\Lambda} (M, N, f) \leq \kdell\!_A M + 1$. Suppose $\kdell\!_A M = m$ is finite. Then there exists $M'\in\mod A$ such that $\Omega_A^m M$ is a stable retract of $\Omega_A^{m+k} M'$. In particular, non-projective summands of $\Omega_A^m M$ are also summands of $\Omega_A^{m+k} M'$, and projective summands of $\Omega_A^m M$ are not necessarily summands of $\Omega_A^{m+k} M'$. We know $\Omega_{\Lambda}^m (M,N,f) = (\Omega_A^m M, 0, 0) \oplus (0,Z,0)$, where $(0,Z,0)$ is infinitely deloopable. If $\Omega_A^m M$ has no projective summand, then $(\Omega_A^m M,0,0)$ is a direct summand of $\Omega_{\Lambda}^{m+k}(M',0,0)=(\Omega_A^{m+k} M', 0,0)$, implying $\kdell\!_{\Lambda}(M,N,f)\leq m$. However, if $\Omega_A^m M$ has projective summands, $(\Omega_A^m M,0,0)$ has a summand of the form $(P,0,0)$ where $P$ is a projective $A$-module. Since $P$ may not be a summand of $\Omega_A^{m+k} M'$, $(P,0,0)$ is not necessarily a stable retract of $(\Omega_A^{m+k} M',0,0)$. Moreover, $(P,0,0)$ does not map to $(0,Z',0)$ nontrivially for any $Z'\in\mod B$ and is not a projective $\Lambda$-module. So, $\kdell\!_{\Lambda} (M, N, f)$ may not be equal to $m$.

We solve this problem by taking another syzygy. Since the non-projective summands of $\Omega_A^m M$ are summands of $\Omega_A^{m+k} M'$, $\Omega_A^{m+1} M$ is a direct summand of $\Omega_A^{m+k+1} M'$. This shows that the only possibly non-infinitely deloopable summand $(\Omega_A^{m+1}M,0,0)$ of $\Omega_{\Lambda}^{m+1}(M,N,f)$ is a summand of $(\Omega_A^{m+k+1}M',0,0)\dirsum \Omega_{\Lambda}^{m+k+1}(M',0,0)$. Therefore, $\kdell\!_{\Lambda} (M, N, f)\leq m+1$.
}
\item From the previous part, we know $\kdell\!_{\Lambda} (0,S',0)=0$ for every simple $B$-module $S'$, so we only need to consider simple $\Lambda$-modules of the form $(S,0,0)$ for every simple $A$-module $S$. Then it follows immediately from the previous part that
{
\[
\sup\{\kdell\!_A S\mid S_A \text{ is simple}\} \leq \sup\{\kdell\!_{\Lambda} (S, 0, 0) \mid S_A \text{ is simple}\} \leq \sup\{\kdell\!_A S\mid S_A \text{ is simple}\} + 1,
\]
which implies $\kdell A \leq \kdell\Lambda \leq \kdell A + 1$.
}
\end{enumerate}
\end{proof}

Now we can show that transforming from $A$ to $\Lambda$ { only changes the derived delooping level by at most 1.}

\begin{proposition}
\label{prop:ddell A vs ddell lambda}
Let $A$ be an Artin algebra and $\Lambda=(\widetilde{A^{\op}})^{\op}$ be as above. { Then $\ddell A \leq \ddell \Lambda \leq \ddell A + 1$.}
\end{proposition}

\begin{proof}
{
We prove $\ddell A\leq \ddell\Lambda$ first. Recall that $\ddell\Lambda$ is always achieved by simple $\Lambda$-modules of the form $(S,0,0)$, where $S$ is a simple $A$-module, since $(0,Z,0)$ is infinitely deloopable for all $Z\in\mod B$.
}

Suppose $\ddell\!_{\Lambda} (S,0,0)=m<\infty$ using $n$ and the exact sequence
\begin{equation}
\label{eq:exact sequence in lambda}
0\to (C_n,D_n,f_n) \to \cdots \to (C_0, D_0, f_0) \to (S,0,0) \to 0,
\end{equation}
where $(i+1)$-$\dell\!_{\Lambda}(C_i,D_i,f_i)\leq m-i$.

By (1) of Lemma \ref{lem:SES in mod lambda and A}, we get the exact sequence in $\mod A$
\begin{equation}
\label{eq:exact sequence in A}
0\to C_n \to \cdots \to C_0 \to S \to 0.
\end{equation}

By Lemma \ref{lem:dell A vs dell lambda} (1), we know { $(i+1)$-$\dell\!_A C_i \leq (i+1)$-$\dell\!_{\Lambda} (C_i,D_i,f_i) \leq m-i$ for $i=1,\dots, n$.} Therefore, $\ddell\!_A S \leq m$. Repeating the argument for all simple $A$-modules $S$ shows $\ddell A\leq \ddell\Lambda$.

On the other hand, if $\ddell A=m$ using $n$ and the exact sequence
\begin{equation}
0\to C_n \to \cdots \to C_0 \to S \to 0
\end{equation}
where $(i+1)$-$\dell\!_A C_i \leq m-i$ for $i=0,\dots, n$, then we can easily induce the following corresponding exact sequence in $\mod\Lambda$
\begin{equation}
\label{eq:induced exact sequence in lambda}
0\to (C_n,0,0) \to \cdots \to (C_0, 0, 0) \to (S,0,0) \to 0,
\end{equation}
where $(i+1)$-$\dell\!_{\Lambda} (C_i,0,0) \leq { (m+1)}-i$ for $i=1,\dots, n$. { Iterating over all simple $A$-modules $S$, we get $\ddell\Lambda\leq\ddell A +1$.}
\end{proof}

{
\begin{remark}
\label{rmk:prop still true with kddell}
Proposition \ref{prop:ddell A vs ddell lambda} is still true if we replace all $\ddell\!$ with the more general $\kddell\!$ due to the same relation between $\kdell\!_A M$ and $\kdell\!_{\Lambda} (M,N,f)$ in Lemma \ref{lem:dell A vs dell lambda}. We include the simpler statement in the proposition for better readability.
\end{remark}
}

We can now prove the main theorem of this section.

\begin{theorem}
\label{thm:main theorem}
The derived delooping level conjecture holds if and only if $\ddell\Lambda=0$ implies $\ddell\Lambda^{\op}<\infty$ for all Artin algebras $\Lambda$.
\end{theorem}

\begin{proof}
The forward direction is trivial, so we prove the reverse direction.

Let $A$ be an Artin algebra. Construct the algebra $\Lambda^{\op}=\widetilde{A^{\op}}$ as before. We know $\ddell\Lambda^{\op}=0$ by construction, so by assumption, $\ddell\Lambda<\infty$. By Proposition \ref{prop:ddell A vs ddell lambda}, $\ddell A<\infty$.
\end{proof}

Another consequence of Proposition \ref{prop:ddell A vs ddell lambda} is that { we can bound $\Findim\Lambda^{\op}$ above in terms of $\ddell A$.} Note that the inequality $\Findim A^{\op} \leq \Findim\Lambda^{\op}$ in Corollary \ref{cor:ddell stays the same after transformation} appeared first in \cite{fossum2006trivial}.

\begin{corollary}
\label{cor:ddell stays the same after transformation}
Let $A$ be an Artin algebra and $\Lambda^{\op}=\widetilde{A^{\op}}$ be as before. { Then $\Findim A^{\op} \leq \Findim\Lambda^{\op} \leq \ddell A+1$. In particular, if $\ddell A=\Findim A^{\op}$, then $\Findim\Lambda^{\op}\in\{\ddell A, \ddell A + 1\}$.}
\end{corollary}

\begin{proof}
By Lemma \ref{lem:SES in mod lambda and A} (1), if $\Findim A^{\op} = \id M = n$ with the minimal injective resolution
\[
0\to M \to I_0 \to \cdots I_n \to 0,
\]
then we get a corresponding injective resolution of $(M,0,0)$ in $\mod\Lambda$
\[
0\to (M,0,0) \to (I_0,0,0) \to \cdots (I_n,0,0) \to 0,
\]
so $\Findim A^{\op} \leq \Findim\Lambda^{\op}$.

Therefore, the corollary follows from
{
\[
\Findim A^{\op} \leq \Findim\Lambda^{\op} \leq \ddell\Lambda \leq \ddell A + 1.
\]
}
\end{proof}

{ We conclude this section with an example where $\dell\Lambda=\ddell\Lambda \neq \dell A=\ddell A$. As we saw in the proof of Proposition \ref{prop:ddell A vs ddell lambda}, the inequality can only occur if there is a simple $A$-module $S$ with $\Omega^{\dell S} S$ having projective summands. This is especially the case if $\gldim A<\infty$.

\begin{example}
Let $A=\mathbb{K} Q/I$ be the quotient of the path algebra of the quiver $A_n$ ($n\geq 2$) with straight orientation
\[
1\to 2\to \cdots\to n,
\]
where $I$ is generated by all paths of length 2. We know $\gldim A = n-1$, achieved by
\begin{equation}
\label{eq:gldim of A_n}
0\to S_n \to \begin{matrix} n-1 \\ n \end{matrix} \to \cdots \to \begin{matrix} 1 \\ 2 \end{matrix} \to S_1 \to 0.
\end{equation}

By construction, the quiver of $\Lambda=(\widetilde{A^{\op}})^{\op}$ is
\[
\begin{tikzcd}
1 \ar[r] \ar[d] & 2 \ar[r] \ar[d] & \cdots \ar[r] & n, \ar[d] \\
1' \ar[loop right] & 2' \ar[loop right] & \cdots & n' \ar[loop right]
\end{tikzcd}
\]
and $\rad\!^2\Lambda=0$. The simple module with the largest delooping level is $S_1$ by the exact sequence
\begingroup
\setlength\arraycolsep{1pt}
\[
0\to \Omega^n S_1 = \bigoplus_{i=1}^n S_{i'} \to \begin{matrix} n \\ n' \end{matrix} \oplus \left( \bigoplus_{i=1}^{n-1} P_{i'}\right) \to \cdots \to \begin{matrix} 1' \\ 1' \end{matrix} \oplus \begin{matrix} & 2 & \\ 3 & & 2' \end{matrix} \to \begin{matrix} & 1 & \\ 2 & & 1' \end{matrix} \to S_1 \to 0,
\]
\endgroup
where $S_{i'}$ is infinitely deloopable for $i=1,2,\dots,n$. The first $n$ syzygies of $S_1$ are $\Omega^j S_1 = S_{j+1} \oplus \left( \bigoplus_{i=1}^j S_{i'} \right)$ for $0<j<n$. Each $S_{j+1}$ as the summand of a $j$-syzygy $ \Omega^j S_1$ is not $(j+1)$-deloopable. Therefore, $\dell\Lambda=n$.

Similarly, we can show $\ddell\!_{\Lambda} S_1$ is equal to $n$ using the exact sequence \eqref{eq:gldim of A_n}. Indeed, $\kdell\!_{\Lambda} \left( \begin{matrix} i \\ i+1 \end{matrix} \right) = \kdell\!_{\Lambda} S_n=1$ for all $k\in\Z_{>0}$ and $i=1,2,\dots,n-1$, since their first syzygy is the direct sum of some infinitely deloopable simple modules from $S_{1'},\cdots,S_{n'}$. Therefore, we obtain
\[
n-1 = \Findim A^{\op} = \dell A = \ddell A < \Findim\Lambda^{\op} = \dell\Lambda = \ddell\Lambda = n. 
\]
\end{example}
}

\section{Tensor product of algebras}
\label{sec:tensor product of algebras}

Let $\mathbb{K}$ be a field and $\Lambda$ be a finite dimensional $\mathbb{K}$-algebra. In the context of the finitistic dimension conjecture, we will assume $\mathbb{K}$ is algebraically closed since the finitistic dimension is invariant under field extensions \cite{jensen1982homological}. We also assume $\Lambda$ is basic since every finite dimensional algebra over $\mathbb{K}$ is Morita equivalent to a basic finite dimensional algebra over $\mathbb{K}$, and the finitistic dimension is invariant under Morita equivalence. In Construction \ref{con: multi-point extension}, if we choose $A=B$ and $\bar{S}=A$ considered as an $A$-bimodule, the triangular matrix algebra $\begin{pmatrix} A & A \\ 0 & A \end{pmatrix}$ is the tensor product $A\otimes_{\mathbb{K}} \mathbb{K}A_2$, where $A_2$ is \begin{tikzcd} 1 \arrow[r] & 2 \end{tikzcd}, the Dynkin quiver of type $A_2$, and the path algebra $\mathbb{K}A_2$ is isomorphic to the $2\times 2$ upper triangular matrix algebra with coefficients in $\mathbb{K}$. In that case, we know the global dimension of the tensor product behaves additively \cite[Lemma 3.4]{xi2000representation}. That is, $\gldim \left( \begin{pmatrix} A & A \\ 0 & A \end{pmatrix} \right) = \gldim A + \gldim \mathbb{K}A_2 = \gldim A + 1$. In general, we are interested in understanding how the derived delooping level behaves under taking tensor product with other finite dimensional algebras over the base field $\mathbb{K}$. We present the first step in proving such results involving the derived delooping level, and more general cases will be the topic of a future paper.

For the rest of this section, let $\Lambda_1$ and $\Lambda_2$ be basic finite dimensional algebras over the algebraically closed field $\mathbb{K}$. Suppose that $\Findim\Lambda_1^{\op}=m<\infty$ and that $\gldim\Lambda_2=n<\infty$ so that $\dell\Lambda_2$ is also $n$ by Observation \ref{prop:kddell is gldim when finite}. We identify $(\Lambda_1\otimes\Lambda_2)^{\op}$ with $\Lambda_1^{\op}\otimes \Lambda_2^{\op}$. If $S_i$ and $T_j$ are simple modules of $\Lambda_1$ and $\Lambda_2$, respectively, then they are 1-dimensional over $\mathbb{K}$. Their tensor product $S_i\otimes_{\mathbb{K}} T_j$ is also 1-dimensional as a $(\Lambda_1\otimes\Lambda_2)$-module, hence simple. If $P_i$ and $Q_j$ are projective modules of $\Lambda_1$ and $\Lambda_2$, respectively, then $P_i\otimes Q_i$ is a projective $(\Lambda_1\otimes \Lambda_2)$-module because it is the direct sum of direct summands of $\Lambda_1\otimes\Lambda_2$.

\begin{lemma}
\label{lem:tensor product}
Suppose $M\in\mod\Lambda_1$ and $N\in\mod\Lambda_2$. Modules denoted with letters $P$ or $Q$ are projective unless stated otherwise.

\begin{enumerate}
\item Given exact sequences $0\to C_m\to\cdots\to C_1\to C_0\to M\to 0$ in $\mod\Lambda_1$ and $0\to D_n \to\cdots\to D_1\to D_0\to N\to 0$ in $\mod\Lambda_2$, there is an exact sequence in $\mod\Lambda_1\otimes \Lambda_2$
\[
0\to C_m\otimes D_n \to \cdots \to \bigoplus_{i+j=k} C_i\otimes D_j \to \cdots \to C_0\otimes D_0 \to M\otimes N \to 0.
\]

In particular, when the two exact sequences are projective resolutions, this shows
\[
\Findim(\Lambda_1^{\op} \otimes\Lambda_2^{\op}) \geq \Findim\Lambda_1^{\op} + \Findim\Lambda_2^{\op}.
\]

\item If $M$ is a direct summand of $M'$, then $M\otimes N$ is a direct summand of $M'\otimes N$.

\item If $\kdell M=m$ and $Q\in\mod\Lambda_2$ is projective, then $\kdell (M\otimes Q)\leq m$ in $\mod(\Lambda_1\otimes \Lambda_2)$.
\end{enumerate}

\end{lemma}

\begin{proof}
\begin{enumerate}
\item By definition.
\item Tensor the split monomorphism $M\to M'$ with the identity map on $N$ to get a split monomorphism $M\otimes N \to M'\otimes N$.
\item Suppose $\Omega^m M\dirsum \Omega^{m+k}M'$ for some $M'$. Applying $-\otimes Q$ to $0\to \Omega^m M \to P_{m-1} \to\cdots \to P_0 \to M\to 0$, we get a truncated projective resolution of $M\otimes Q$ whose $m$-syzygy is $\Omega^m M\otimes Q$. By the previous part, it is a direct summand of the $(m+k)$-syzygy $\Omega^{m+k}M'\otimes Q$.
\end{enumerate}
\end{proof}

\begin{proposition}
\label{prop:ddell of tensor}
Let $S$ and $T$ be simple modules over $\Lambda_1$ and $\Lambda_2$, respectively, and $\pd T = t$. If the $s$-syzygy $\Omega^s S$ can be delooped $t+1$ more times, \textit{i.e.}, $(s+t+1)$-$\dell \Omega^s S = 0$, then $\ddell(S\otimes T)\leq s+t$.
\end{proposition}

\begin{proof}
We have exact sequences
\[
0\to \Omega^s S \to P_{s-1} \to \cdots \to P_1 \to P_0 \to S \to 0,
\]
\[
0\to \Omega^t T=Q_t \to Q_{t-1} \to \cdots \to Q_1 \to Q_0 \to T \to 0,
\]
where all $P's$ and $Q's$ are projective.

By taking their tensor product as shown in Lemma \ref{lem:tensor product} (1), we get an exact sequence
\begin{equation}
\label{eq: tensor projective resolution}
{ 0\to \Omega^s S\otimes Q_t \to (P_{s-1}\otimes Q_t) \oplus (\Omega^s S \otimes Q_{t-1}) \to \cdots \to Q_0\otimes P_0 \to S\otimes T\to 0},
\end{equation}
where the first non-projective term counting from the right is $\Omega^s S \otimes Q_0$ { if $\Omega^s S$ is not projective. If $\Omega^s S$ is projective, then every term in \eqref{eq: tensor projective resolution} is projective, and the result follows.}

So assume $\Omega^s S$ is not projective. All non-projective summands of the terms in \eqref{eq: tensor projective resolution} are at positions $s+t$ to $s$ ($Q_0\otimes P_0$ is at position 0), and they are $\Omega^s S\otimes Q_t$, $\Omega^s S\otimes Q_{t-1}$, $\dots$, $\Omega^s S\otimes Q_0$. Since $(s+t+1)$-$\dell \Omega^s S = 0$, we have $\kdell \Omega^s S = 0$ for $k=s+1,\dots, s+t+1$. By Lemma \ref{lem:tensor product} (3), { $k$-$\dell (\Omega^s S\otimes Q) = 0$ for any projective $\Lambda_2$-module $Q$} for $k=s+1,\dots, s+t+1$. Therefore, we get
\begin{itemize}
\item $(s+t+1)$-$\dell (\Omega^s S\otimes Q_t) = 0 \leq { s+t-(s+t)=0,}$
\item $(s+t)$-$\dell (\Omega^s S\otimes Q_{t-1}) = 0 \leq { s+t-(s+t-1)=1,}$

\hspace{-0.8cm}\vdots
\item $(s+1)$-$\dell (\Omega^s S\otimes Q_0)= 0\leq { s+t-s=t.}$
\end{itemize}

The terms at other positions of the exact sequence are all projective, so we naturally have their $k$-delooping level zero for any $k$. Therefore, by the definition of $\ddell\!$, we get $\ddell (S\otimes T) \leq s+t$.
\end{proof}

\begin{corollary}
If $\gldim\Lambda_2=n<\infty$ and $\Findim\Lambda_1^{\op} = \kdell\Lambda_1 = m$ for any $k\geq n+1$ , then
\[
\ddell\Lambda_1 + \ddell\Lambda_2 = \ddell(\Lambda_1\otimes \Lambda_2) = \Findim(\Lambda_1\otimes \Lambda_2)^{\op}.
\]
\end{corollary}

\begin{proof}
The corollary follows from
\begin{align*}
m+n & \geq \ddell(\Lambda_1\otimes \Lambda_2) \geq \Findim(\Lambda_1\otimes\Lambda_2)^{\op} = \Findim(\Lambda_1^{\op} \otimes\Lambda_2^{\op}) \geq \Findim\Lambda_1^{\op} + \Findim\Lambda_2^{\op} \\
& = \ddell\Lambda_1 + \ddell\Lambda_2 = m+n,
\end{align*}
where in particular the first inequality $\ddell(\Lambda_1\otimes \Lambda_2)\leq m+n$ is a consequence of Proposition \ref{prop:ddell of tensor} since the argument works for any pair of simple modules $S$ and $T$.
\end{proof}

\begin{corollary}
If $\Findim\Lambda_1^{\op} = \kdell\Lambda_1 = m < \infty$ for all $k\in\Z_{>0}$, then
\[
\ddell\Lambda_1 + \ddell\Lambda_2 = \ddell(\Lambda_1\otimes \Lambda_2) = \Findim(\Lambda_1\otimes \Lambda_2)^{\op}.
\]
for any $\Lambda_2$ with finite global dimension.
\end{corollary}

\section{Future Directions}

We hope the new definition of derived delooping level may rekindle more interest in studying the finitistic dimensions. As we saw in Proposition \ref{prop:ddell A vs ddell lambda}, the derived delooping level does not change under the construction of $\Lambda$ from $A$. It is important to consider under what other constructions and operations does the derived delooping level stay unchanged. We formulate some future questions below.

\begin{question}
\begin{enumerate}
\item Can we extend Proposition \ref{prop:ddell A vs ddell lambda} to other constructions of $\Lambda$ from $A$ that are related to triangular matrix algebras?
\item Can we loosen the condition of $\Lambda_2$ in Section \ref{sec:tensor product of algebras} to be more general algebras, such as algebras $\Lambda$ whose $\Findim\Lambda^{\op}$ is $\dell\Lambda$ or $\ddell\Lambda$?
\item Can we prove more general versions of Proposition \ref{prop:ddell of tensor} where the conditions on $S$ and $T$ are weaker? For example, what if $T$ is infinitely deloopable?
\end{enumerate}
\end{question}

\section{Declarations}

\textbf{Ethical Approval:} Not applicable

\textbf{Funding:} Not applicable

\textbf{Availability of Data and Materials:} Not applicable

\bibliographystyle{plain}
\bibliography{refs_v3}

\end{document}